\documentclass[a4,12pt]{article}
\usepackage{amsmath,amsthm,amssymb,titlefoot}

\theoremstyle{plain}
\newtheorem{df}{Definition}[section]
\newtheorem{thm}[df]{Theorem}
\newtheorem{lem}[df]{Lemma}
\newtheorem{prop}[df]{Proposition}

\theoremstyle{definition}
\newtheorem{rem}{Remark}

\makeatletter

\@addtoreset{equation}{section}
\makeatother

\allowdisplaybreaks

\newcommand{\ind}{{\bf 1}}
\newcommand{\R}{\boldsymbol{R}}
\newcommand{\B}{\mathcal{B}}
\newcommand{\F}{\mathcal{F}}
\newcommand{\cH}{\mathcal{H}}
\newcommand{\cS}{\mathcal{S}}
\newcommand{\cP}{\mathcal{P}}
\newcommand{\as}{\mbox{{\it a}.{\it s}.}}
\newcommand{\cae}{\mbox{{\it a}.{\it e}.}}
\newcommand{\norm}[1]{\Vert #1 \Vert}
\newcommand{\sgn}[1]{{\rm sgn}(#1)}

\begin{document}
\title{The $L^p$ Cauchy sequence for one-dimensional BSDEs with linear growth generators}
\author{Yuki Izumi
\thanks{Graduate School of Mathematics, Kyushu University, Fukuoka, Japan.}
}
\authorfootnote{E-mail: {\tt ma210008@math.kyushu-u.ac.jp}}
\keywords{backward stochastic differential equation, $L^p$ solution, linear growth generator}
\amssubj{60H10}
\date{}
\maketitle

\begin{abstract}
In this paper, the existence of $L^p~(p>1)$ solutions for one-dimensional backward stochastic differential equations will be shown directly by proving that an approximation sequence is a Cauchy one in the $L^p$ sense. 
\end{abstract}

\section{Introduction}%%%%%%%%%%%%%%%%%%%%%%%%%%%%%%%
In this paper, we consider the following one-dimensional backward stochastic differential equation (BSDE in short): 
\begin{align*}%\label{BSDE0}
\left\{
 \begin{array}{lr}
 -dY_t=f(t,Y_t,Z_t)dt-Z_t\cdot dW_t, & 0\leq t \leq T, \\
 Y_T=\xi, 
 \end{array}
\right.
\end{align*}
where $T>0$, $\xi$ is a random variable, $f$ is a real-valued random function, and $W$ is a $d$-dimensional Brownian motion with $W_0=0$. 
The function $f$ is called the generator. 
The equation above is also written in
\begin{align}\label{BSDE}
Y_t=\xi+\int_t^T f(s,Y_s,Z_s)ds-\int_t^T Z_s\cdot dW_s, \qquad 0 \leq t\leq T. 
\end{align}
A pair $(Y,Z)$ of adapted processes satisfying the equation is called a solution. 

As for $L^p~(p>1)$ solutions to the BSDE, El Karoui et al. \cite{Karoui} proved an existence and uniqueness result when $f$ is Lipschitz continuous and $\xi$ is in $L^p$ by using a fixed-point theorem. 
A natural question then arises whether the Lipschitz condition can be relaxed. 
On account of the standard forward SDEs, the linear growth condition of the generator seems to be a candidate for a weaker condition to guarantee the existence and the $L^p$-integrability of solutions. 
Hereinafter, we assume that $f$ is continuous and of linear growth order and $\xi$ is in $L^p$. 
In this case, the existence results were shown by Lepeltier and San Martin \cite{Lepeltier} for $p=2$, by Chen \cite{Chen} for $1<p\leq 2$ and after them by Fan and Jiang \cite{Fan} for general $p>1$. 
In these papers, a key role is played by an approximation sequence. 
When $1<p\leq 2$, the existence was obtained by proving that the sequence is a Cauchy one. 
When $p>2$, an $L^p$ solution was constructed by taking advantage of a stopping time argument. 
And, it remains open to prove the sequence to be a Cauchy one when $p>2$. 

This paper is organized as follows. 
In Section 2, a priori estimates are obtained by using It\^{o}'s formula. 
In Section 3, the approximation sequence is constructed. 
Then, it is proved that the sequence is a Cauchy one and converges to an $L^p$ solution to the BSDE \eqref{BSDE}. 

\section{Preliminaries}%%%%%%%%%%%%%%%%%%%%%%%%%%%%%
%%%%%%%%%%%%%%%%%%%%%%%%%%%
\subsection{Notations}
Let $(W_t)_{0 \leq t \leq T}$ be a $d$-dimensional Brownian motion with $W_0=0$ defined on a complete probability space $(\Omega,\F,P)$, and $(\F_t)_{0 \leq t \leq T}$ be the natural filtration of the Brownian motion $W$ augmented by the $P$-null sets of $\F$.
Throughout the paper, we are working on with only one filtration $(\F_t)$ and for the sake of simplicity, we omit the prefix ``$(\F_t)$-''; for example, we just say ``adapted'' instead of ``$(\F_t)$-adapted''.
We denote by $\cP$ the predictable sub-$\sigma$-field of $\B([0,T])\otimes \F$, and let the generator $f$, which is defined on $[0,T]\times \Omega \times \R \times \R^d$, be $\cP\otimes\B(\R^{d+1})/\B(\R)$-measurable.
For a given $p>1$, we denote by $\cS^p$ the set of real-valued, continuous and adapted processes $(\eta_t)_{0\leq t \leq T}$ such that
\[
\norm{\eta}_{\cS^p}:=\left\{E\left[\sup_{0\leq t\leq T}|\eta_t|^p\right]\right\}^{1/p}<\infty.
\]
$\cH^p$ stands for the set of $\R^d$-valued predictable processes $(\zeta_t)_{0 \leq t \leq T}$ such that
\[
\norm{\zeta}_{\cH^p}:=\left\{E\left[\left( \int_0^T |\zeta_s|^2ds \right)^\frac{p}{2}\right]\right\}^{1/p}<\infty.
\]
We see that the following properties hold:
\begin{itemize}
\item If $\norm{\eta^n-\eta^m}_{\cS^p}\to 0$ as $n,m\to\infty$, then there exists a unique $\eta\in\cS^p$ such that $\norm{\eta^n-\eta}_{\cS^p}\to 0$ as $n\to\infty$,
\item if $\norm{\zeta^n-\zeta^m}_{\cH^p}\to 0$ as $n,m\to\infty$, then there exists a unique $\zeta\in\cH^p$ such that $\norm{\zeta^n-\zeta}_{\cH^p}\to 0$ as $n\to\infty$.
\end{itemize}

%%%%%%%%%%%%%%%%%%%%%%%%%%%%%%%%%%%%%%%%%%%%%%
\subsection{Assumptions}
In this paper, we use the following assumptions (H1)-(H3):
\begin{description}
\item[\rm (H1)] There exists a positive constant $K$ and a non-negative predictable process $(g_t)_{0 \leq t \leq T}$ such that
                \begin{multline*}
                E\left[ \left( \int_0^T g_s ds \right)^p \right]<\infty, \quad
                |f(t,\omega,y,z)|\leq g_t(\omega)+K(|y|+|z|)\\
                \mbox{for any}~(t,\omega,y,z)\in [0,T]\times \Omega\times \R \times \R^d. 
                \end{multline*}
\item[\rm (H2)] For each $(t,\omega)\in [0,T]\times \Omega$, $f(t,\omega,y,z)$ is continuous in $(y,z)$. 
\item[\rm (H3)] $\xi\in L^p$, i.e., $E[|\xi|^p]<\infty$.
\end{description}

\begin{df}
A solution to the BSDE with the generator $f$ and the terminal value $\xi$ is a pair of continuous adapted processes $Y$ and predictable processes $Z$ such that
\begin{align*}
&\int_0^T \big\{ |f(s,Y_s,Z_s)|+|Z_s|^2 \big\}ds<\infty \quad \as
\end{align*}
and satisfies \eqref{BSDE}.
In particular, we call a solution $(Y,Z)\in \cS^p\times\cH^p$ an $L^p$ solution to the BSDE.
\end{df}

In the case $p>1$ and the generator is Lipschitz, the existence and uniqueness of $L^p$ solution is known (\cite{Karoui}).
\begin{thm}\label{EU}
Assume that $f$ is uniformly Lipschitz in $(y,z)$, i.e., there exists a positive constant $C$ such that
\begin{multline*}
|f(t,\omega,y_1,z_1)-f(t,\omega,y_2,z_2)|\leq C(|y_1-y_2|+|z_1-z_2|) \\
\mbox{for any}~~(t,\omega)\in [0,T]\times \Omega,~y_1,y_2\in\R,~z_1,z_2\in\R^d. 
\end{multline*}
And assume (H3) holds and \[E\left[  \left( \int_0^T |f(s,0,0)|ds \right)^p \right]<\infty. \]
Then, BSDE \eqref{BSDE} has a unique $L^p$ solution.
\end{thm}
It is also known (\cite{Karoui}) that
\begin{thm}\label{CT}
For $i=1,2$, let $f^i$ be uniformly Lipschitz in $(y,z)$, $\xi^i$ satisfy (H3) and
\[
E\left[ \left( \int_0^T |f^i(s,0,0)|ds \right)^p \right]<\infty. 
\]
In addition, assume that each $(Y^i,Z^i)$ is the $L^p$ solution to the BSDE with respect to $(f^i,\xi^i)$.
Then, $\xi^1 \geq \xi^2~\as$ and $f^1(t,Y^2_t,Z^2_t)\geq f^2(t,Y^2_t,Z^2_t)~dt\times dP\mbox{-\cae}$ imply $Y^1 \geq Y^2~\as$.
\end{thm}

\begin{rem}
In \cite{Karoui}, the assertion of Theorem \ref{EU} and \ref{CT} are stated under the assumptions like
\begin{align}
E\left[ \left( \int_0^T |f(s,0,0)|^2ds \right)^\frac{p}{2} \right]<\infty, \label{rem1}
\end{align}
which is stronger than the ones in the theorems. 
Observing the proof in \cite{Karoui} carefully, we can weaken the assumption \eqref{rem1} to the one as we used.
\end{rem}

%%%%%%%%%%%%%%%%%%%%%%%%%%%%%%%%
\subsection{A priori estimates}
%--------------------------
We prepare the following estimations which play a key role in the observation of this paper, 
by generalizing the ones in \cite{Chen} used by Chen for specified solutions.
%--------------------------
\begin{prop}\label{APE}%\label{pr2.3}
\begin{description}
\item[\rm (i)] Let $p>1$. If $(Y,Z)$ is an $L^p$ solution to the BSDE \eqref{BSDE}, then there exists a positive constant $C_p$ depending only on $p$ such that
\begin{align*}
&\norm{Y}^p_{\cS^p}\leq C_p E\left[ |\xi|^p+\int_0^T |Y_s|^{p-1}|f(s,Y_s,Z_s)|ds \right],\\[.5em]
&\norm{Z}^p_{\cH^p}\leq C_p \left\{ E\left[ |\xi|^p+\left( \int_0^T |Y_s||f(s,Y_s,Z_s)|ds \right)^{\frac{p}{2}} \right]+\norm{Y}^p_{\cS^p} \right\}.
\end{align*}
Moreover, if $f$ satisfies (H1), then
\begin{align*}
\norm{Z}^p_{\cH^p}&\leq C (1+\norm{Y}^\frac{p}{2}_{\cS^p}+\norm{Y}^p_{\cS^p}),
\end{align*}
where $C$ is a positive constant which depends only on $p,K,T,E[|\xi|^p]$ and $E[(\int_0^T g_s ds)^p]$. 
\item[\rm (ii)] Let $p>1$. If $(Y^i,Z^i)$ is an $L^p$ solution to the BSDE with respect to $(f^i,\xi^i)$, $i=1,2$, respectively, then there exists a positive constant $C_p$ depending only on $p$ such that
\begin{align*}
& \norm{\delta Y}^p_{\cS^p} \leq C_p E\left[ |\delta Y_T|^p+\int_0^T |\delta Y_s|^{p-1}|\delta f_s|ds \right], \\
& \norm{\delta Z}^p_{\cH^p} \leq C_p \left\{ E\left[ |\delta Y_T|^p+\left( \int_0^T |\delta Y_s||\delta f_s|ds \right)^\frac{p}{2} \right]+\norm{\delta Y}^p_{\cS^p} \right\}, 
\end{align*}
where $\delta Y:=Y^1-Y^2,~\delta Z:=Z^1-Z^2,~\delta f_s:=f^1(s,Y^1_s,Z^1_s)-f^2(s,Y^2_s,Z^2_s)$.
\end{description}
\end{prop}
%proof--------------
\begin{proof}
The assertion (ii) follows from (i). 
Namely, put $\tilde{f}(t,y,z)=f^1(t,Y^2_t+y,Z^2_t+z)-f^2(t,Y^2_t,Z^2_t)$. 
Then, $\delta f_t=\tilde{f}(t,\delta Y_t,\delta Z_t)$ and the pair $(\delta Y,\delta Z)\in \cS^p \times \cH^p$ satisfies \begin{align*}
\delta Y_t &= \delta Y_T + \int_t^T \tilde{f}(s,\delta Y_s,\delta Z_s)ds - \int_t^T \delta Z_s \cdot dW_s, \quad 0\leq t \leq T. 
\end{align*}
Thus, we only prove (i). 
\par
Let $p>1$. 
We first estimate $Y$. 
As an elementary application of It\^{o}'s formula, we obtain
\begin{multline}
|Y_t|^p+\frac{p(p-1)}{2}\int_t^T |Y_s|^{p-2}\tilde{\ind}(Y_s)|Z_s|^2ds\\
=|\xi|^p+p\int_t^T \sgn{Y_s}|Y_s|^{p-1}f(s,Y_s,Z_s)ds \\
-p\int_t^T \sgn{Y_s}|Y_s|^{p-1}Z_s\cdot dW_s,\qquad 0\leq t \leq T, \label{yp}
\end{multline}
where
\begin{align*}
&\tilde{\ind}(y):=\left\{
 \begin{array}{ll}
 \ind_{\{y\neq 0\}},&1<p<2\\[.5em]
 1,&2\leq p
 \end{array}
\right.,
\qquad
\sgn{x}:=\left\{
 \begin{array}{rl}
  -1,&x<0\\
  0,&x=0\\
  1,&x>0
 \end{array}
 \right..
\end{align*}
See also \cite[Lemma 2.2]{Briand}. 
Hence, we get
\begin{multline}\label{eq2.2}
\sup_{0\leq t \leq T}|Y_t|^p \leq |\xi|^p+p\int_0^T |Y_s|^{p-1}|f(s,Y_s,Z_s)|ds \\
+2p\sup_{0\leq t \leq T}\left| \int_0^t\sgn{Y_s}|Y_s|^{p-1}Z_s\cdot dW_s \right|.
\end{multline}
By the Burkholder-Davis-Gundy inequality (the BDG inequality in short), there exists a positive constant $C_1$ such that
\begin{align}
&2pE\left[ \sup_{0\leq t \leq T}\left| \int_0^t\sgn{Y_s}|Y_s|^{p-1}Z_s\cdot dW_s \right| \right]\nonumber\\
&\qquad \leq 2pC_1E\left[ \left( \int_0^T |Y_s|^{2p-2}\tilde{\ind}(Y_s)|Z_s|^2ds \right)^{\frac{1}{2}} \right] \nonumber\\
&\qquad \leq 2pC_1E\left[ \sup_{0\leq t \leq T}|Y_t|^{\frac{p}{2}} \left( \int_0^T |Y_s|^{p-2}\tilde{\ind}(Y_s)|Z_s|^2 ds\right)^{\frac{1}{2}} \right] \nonumber\\
&\qquad \leq \frac{1}{2}E\left[ \sup_{0\leq t \leq T}|Y_t|^p \right]+2p^2C_1^2E\left[ \int_0^T |Y_s|^{p-2}\tilde{\ind}(Y_s)|Z_s|^2ds \right], \label{eq2.4}
\end{align}
where, to see the third inequality above, we have used the inequality
\begin{align*}\label{epsilon}
2ab\leq \varepsilon a^2+\varepsilon^{-1}b^2,\qquad \varepsilon>0,~~a,b\geq 0 \tag{$*$}
\end{align*}
with $\varepsilon=1/2$.

By the H\"{o}lder inequality, we have
\begin{align*}
&E\left[ \left( \int_0^T |Y_s|^{2p-2}\tilde{\ind}(Y_s)|Z_s|^2ds \right)^{\frac{1}{2}} \right]\\
&\qquad \leq E\left[ \sup_{0\leq t \leq T}|Y_t|^{p-1} \left( \int_0^T |Z_s|^2ds \right)^\frac{1}{2} \right]\\
&\qquad \leq \left\{ E\left[ \sup_{0\leq t \leq T}|Y_t|^p \right] \right\}^{1-\frac{1}{p}} \left\{ E\left[ \left( \int_0^T |Z_s|^2ds \right)^\frac{p}{2} \right] \right\}^\frac{1}{p} <\infty.
\end{align*}
Thus, $(\int_0^t\sgn{Y_s}|Y_s|^{p-1}Z_s\cdot dW_s)_{0\leq t \leq T}$ is a martingale. Then, taking the expectations of \eqref{yp}, we get
\begin{multline}
\frac{p(p-1)}{2}E\left[ \int_0^T |Y_s|^{p-2}\tilde{\ind}(Y_s)|Z_s|^2ds \right]\\ 
\leq E\left[ |\xi|^p+p\int_0^T |Y_s|^{p-1}|f(s,Y_s,Z_s)|ds \right].\label{eq2.5}
\end{multline}
Then \eqref{eq2.2}, \eqref{eq2.4} and \eqref{eq2.5} yield the estimation of $Y$. 

Next, we estimate $Z$. 
By \eqref{yp} with $p=2$, we deduce that
\begin{align*}
\int_0^T |Z_s|^2ds&\leq |\xi|^2+2\int_0^T |Y_s||f(s,Y_s,Z_s)|ds+2\sup_{0\leq t \leq T}\left| \int_0^t Y_sZ_s\cdot dW_s \right|.
\end{align*}
Hence, it follows that
\begin{multline}\label{eq2.6}
\left( \int_0^T |Z_s|^2ds \right)^\frac{p}{2} \\
\leq C_2 \left\{ |\xi|^p  + \left( \int_0^T |Y_s||f(s,Y_s,Z_s)|ds \right)^\frac{p}{2}  + \sup_{0\leq t \leq T}\left| \int_0^t Y_sZ_s\cdot dW_s \right|^\frac{p}{2} \right\}, 
\end{multline}
where $C_2$ is a positive constant depending only on $p$. 
By the BDG inequality, there exists a positive constant $C_3$ depending only on $p$ such that
\begin{align}
&C_2E\left[ \sup_{0\leq t \leq T}\left| \int_0^t Y_sZ_s\cdot dW_s \right|^\frac{p}{2} \right] \nonumber\\
&\qquad \leq C_2 C_3 E\left[ \left( \int_0^T |Y_s|^2|Z_s|^2ds \right)^\frac{p}{4} \right] \nonumber\\
&\qquad \leq C_2C_3 E\left[ \sup_{0 \leq t \leq T}|Y_t|^\frac{p}{2} \left( \int_0^T |Z_s|^2ds \right)^\frac{p}{4} \right] \nonumber\\
&\qquad \leq 2C_2^2C_3^2 E\left[ \sup_{0 \leq t \leq T}|Y_t|^p \right]+\frac{1}{2}E\left[ \left( \int_0^T |Z_s|^2ds \right)^\frac{p}{2} \right], \label{eq2.7}
\end{align}
where, to see the third inequality above, we have used \eqref{epsilon} again with $\varepsilon=1/2$. 
Then, we get the second estimation from \eqref{eq2.6} and \eqref{eq2.7}. 

We finally show the last assertion of (i). 
To do this, it is sufficient to estimate the second term of the estimation with respect to $Z$. 
By (H1) and the H\"{o}lder inequality, there exists positive constants $C_{p,K},C_{p,K,T}$ and $C'_{p,K,T}$ which depend only on the subscripts such that
\begin{align*}
&E\left[ \left( \int_0^T |Y_s||f(s,Y_s,Z_s)|ds \right)^\frac{p}{2} \right] \\
&\leq C_{p,K} \left\{ E\left[ \left( \int_0^T |Y_s|g_s ds \right)^\frac{p}{2} \right] \right.\\
& \hspace{4.5em}\left. + E\left[ \left( \int_0^T |Y_s|^2ds \right)^\frac{p}{2} \right] + E\left[ \left( \int_0^T |Y_s||Z_s|ds \right)^\frac{p}{2} \right] \right\} \\
&\leq C_{p,K,T} \left\{ \norm{Y}^\frac{p}{2}_{\cS^p} \left\{ E\left[ \left( \int_0^T g_sds \right)^p \right] \right\}^\frac{1}{2} \right.\\
& \hspace{5.5em}\left. + \norm{Y}^p_{\cS^p} + E\left[ \left( \int_0^T \left( \varepsilon^{-1}|Y_s|^2 + \varepsilon|Z_s|^2 \right)ds \right)^\frac{p}{2} \right] \right\} \\
&\leq C'_{p,K,T} \left( \norm{Y}^\frac{p}{2}_{\cS^p} \left\{ E\left[ \left( \int_0^T g_sds \right)^p \right] \right\}^\frac{1}{2} + \varepsilon^{-\frac{p}{2}}\norm{Y}^p_{\cS^p} + \varepsilon^\frac{p}{2} \norm{Z}^p_{\cH^p} \right), 
\end{align*}
where, to see the second inequality above, we have used \eqref{epsilon} with $C_pC'_{p,K,T}\varepsilon^\frac{p}{2}=1/2$. 
Then, we obtain the desired estimation. 
\end{proof}

\section{Existence of an $L^p$ solution}%%%%%%%%%%%%%%%%%%%%%%%%%%%%%%
%%%%%%%%%%%%%%%%%%%%%%%%%%%%%%%%%%%%%%%%%%%%%%%%%%%%%%%
\subsection{Approximation of linear growth functions}
According to \cite{Lepeltier}, linear growth functions can be approximated by Lipschitz functions.
Precisely speaking, when a generator $f$ satisfies (H1) and (H2), 
\begin{align}
f_n(t,y,z)&:=\inf_{(u,v)\in \R^{d+1}}\{f(t,u,v)+n(|y-u|+|z-v|)\},\qquad n \geq K \label{kinji}
\end{align}
is a Lipschitz function and approximates the linear growth function $f$, where $K$ is a constant appeared in (H1). 

\begin{lem}\label{le3.1}
Assume (H1) and (H2) hold. Then, \eqref{kinji} is well-defined for $n\geq K$ and the following properties i)-iv) hold:
\begin{description}
\item[\rm i)] $|f_n(t,\omega,y,z)|\leq g_t(\omega)+K(|y|+|z|) ~~\mbox{for any}~~ (t,\omega,y,z)\in [0,T]\times \Omega \times \R \times \R^d$,
\item[\rm ii)] $f_n \leq f_{n+1} \leq f,\quad n \geq K$,
\item[\rm iii)] $|f_n(t,\omega,y_1,z_1)-f_n(t,\omega,y_2,z_2)|\leq n(|y_1-y_2|+|z_1-z_2|) ~~\mbox{for any}~~ (t,\omega)\in [0,T]\times \Omega$,
\item[\rm iv)] if $(y_n,z_n)\to (y,z)$, then $f_n(t,\omega,y_n,z_n)\to f(t,\omega,y,z)  ~~\mbox{for any}~~ (t,\omega)\in [0,T]\times \Omega$.
\end{description}
\end{lem}

%%%%%%%%%%%%%%%%%%%%%%%%%%%%%%%%%%%%%%%%%
\subsection{Approximation of a solution}
Let $p>1$ and assumptions (H1)-(H3) hold. 
We consider the following one-dimensional BSDEs:
\begin{align}
&Y^n_t=\xi+\int_t^T f_n(s,Y^n_s,Z^n_s)ds-\int_t^T Z^n_s\cdot dW_s, \qquad n\geq K, \label{ABSDE}\\
&U_t=\xi+\int_t^T \{g_s+K(|U_s|+|V_s|)\}ds-\int_t^T V_s\cdot dW_s. \nonumber
\end{align}
Theorem \ref{EU} assures the existence and uniqueness of $L^p$ solution to these BSDEs. 
Thus, $(Y^n,Z^n)$ and $(U,V)$ are well-defined for $n \geq K$.
Moreover, by Theorem \ref{CT} and Lemma \ref{le3.1}-ii), we have
\begin{align}\label{eq3.2}
Y^n \leq Y^{n+1} \leq U, \qquad n\geq K.
\end{align}

\begin{thm}\label{pr3.3}
$(Y^n,Z^n)$ is a Cauchy sequence in $\cS^p \times \cH^p$.
\end{thm}
%proof-------------------
\begin{proof}
The assertion for $1<p\leq 2$ can be proved in the same manner as \cite[Lemma 4]{Chen}. 
Thus, we give the proof only for the case $p>2$. 

Since $(Y^n)$ is non-decreasing, it admits the limit process $Y$. 
By \eqref{eq3.2}, it follows that
\begin{align*}
Y^{\lceil K \rceil} \leq Y^n, Y \leq U, \qquad n\geq K, 
\end{align*}
where $\lceil \cdot \rceil$ represents the ceiling function. 
Thus, we have
\begin{align}\label{eqN3.4}
|Y^n_{\cdot}|\leq M, \quad |Y_{\cdot}|\leq M, \qquad n\geq K, 
\end{align}
where $\sup_{0\leq t \leq T}|Y^{\lceil K \rceil}_t| \vee \sup_{0\leq t \leq T}|U_t|=:M \in L^p$. 
Then, by the dominated convergence theorem, it follows that
\begin{align*}
&E\left[ \int_0^T |Y^n_s-Y_s|^{p-1}g_sds \right] \to 0, \quad E\left[ \int_0^T |Y^n_s-Y_s|^pds \right] \to 0,
\end{align*}
and thus, we get
\begin{multline}
E\left[ \int_0^T |Y^n_s-Y^m_s|^{p-1}g_sds \right] \to 0, \quad E\left[ \int_0^T |Y^n_s-Y^m_s|^pds \right] \to 0, \\\mbox{as ~$n,m\to\infty$}. \label{eq3.3}
\end{multline}

By Proposition \ref{APE}-(ii), we have
\begin{align}
& \norm{Y^n-Y^m}^p_{\cS^p} \nonumber\\
& \hspace{2em}\leq C_pE\left[ \int_0^T |Y^n_s-Y^m_s|^{p-1}|f_n(s,Y^n_s,Z^n_s)-f_m(s,Y^m_s,Z^m_s)|ds \right], \label{A1}\\
& \norm{Z^n-Z^m}^p_{\cH^p} \nonumber\\
& \hspace{2em}\leq C_p \left\{ E\left[ \left( \int_0^T |Y^n_s-Y^m_s||f_n(s,Y^n_s,Z^n_s)-f_m(s,Y^m_s,Z^m_s)|ds \right)^\frac{p}{2} \right]\right. \nonumber\\
& \hspace{5.5em}\left.\vphantom{\left( \int_0^T |Y^n_s-Y^m_s||f_n(s,Y^n_s,Z^n_s)-f_m(s,Y^m_s,Z^m_s)|ds \right)^\frac{p}{2}}+\norm{Y^n-Y^m}^p_{\cS^p} \right\}. \label{A2}
\end{align}
We first estimate the right hand side of \eqref{A1}. 
By Lemma \ref{le3.1}-i), we get
\begin{multline}
E\left[ \int_0^T |Y^n_s-Y^m_s|^{p-1}|f_n(s,Y^n_s,Z^n_s)-f_m(s,Y^m_s,Z^m_s)|ds \right] \\
\leq 2E\left[ \int_0^T |Y^n_s-Y^m_s|^{p-1}g_sds \right]+ K E\left[ \int_0^T |Y^n_s-Y^m_s|^{p-1} F_{n,m}(s)ds \right], \label{eq3.6}
\end{multline}
where $F_{n,m}(s):=|Y^n_s|+|Z^n_s|+|Y^m_s|+|Z^m_s|$.
By \eqref{eq3.3}, we know the first term of \eqref{eq3.6} converges to zero. 
Thus, we estimate the second term of this.
By the H\"{o}lder inequality and \eqref{epsilon}, we have
\begin{align}
& K E\left[ \int_0^T |Y^n_s-Y^m_s|^{p-1} F_{n,m}(s)ds \right] \nonumber\\
&\leq K E\left[ \left( \int_0^T |Y^n_s-Y^m_s|^{2p-2}ds \right)^\frac{1}{2} \left( \int_0^T \{ F_{n,m}(s) \}^2ds \right)^\frac{1}{2} \right] \nonumber\\
&\leq K E\left[ \sup_{0 \leq t \leq T}|Y^n_t-Y^m_t|^\frac{p}{2} \left( \int_0^T |Y^n_s-Y^m_s|^{p-2}ds \right)^\frac{1}{2} \left( \int_0^T \{ F_{n,m}(s) \}^2ds \right)^\frac{1}{2} \right] \nonumber\\
&\leq \varepsilon E\left[ \sup_{0 \leq t \leq T}|Y^n_t-Y^m_t|^p \right]+\varepsilon^{-1}K^2 E\left[ \int_0^T |Y^n_s-Y^m_s|^{p-2}ds \int_0^T \{ F_{n,m}(s) \}^2ds \right] \nonumber\\
&\leq \varepsilon \norm{Y^n-Y^m}^p_{\cS^p} \nonumber\\
&\hspace{3em}+\varepsilon^{-1} K^2 \left\{ E\left[ \left( \int_0^T |Y^n_s-Y^m_s|^{p-2}ds \right)^\frac{p}{p-2} \right] \right\}^{1-\frac{2}{p}} \nonumber\\
& \hspace{12em}\times\left\{ E\left[ \left( \int_0^T \{F_{n,m}(s)\}^2ds \right)^\frac{p}{2} \right] \right\}^\frac{2}{p} \nonumber\\
&\leq \varepsilon \norm{Y^n-Y^m}^p_{\cS^p} \nonumber\\
&\hspace{3em}+\varepsilon^{-1}K^2T^\frac{2}{p}\left\{ E\left[ \int_0^T |Y^n_s-Y^m_s|^pds \right] \right\}^{1-\frac{2}{p}} \nonumber\\
&\hspace{12em}\times\left\{ E\left[ \left( \int_0^T \{F_{n,m}(s)\}^2ds \right)^\frac{p}{2}\right] \right\}^\frac{2}{p}. \label{eq3.7}
\end{align}
By \eqref{eqN3.4}, we have
\begin{align*}
\sup_{n \geq K}\norm{Y^n}_{\cS^p}<\infty. 
\end{align*}
Thus, by Proposition \ref{APE}-(i), we see that
\begin{align*}
\sup_{n,m\geq K}E\left[ \left( \int_0^T \{F_{n,m}(s)\}^2ds \right)^\frac{p}{2} \right]<\infty.
\end{align*}
Letting $\varepsilon$ such that $C_p\varepsilon=1/2$, by \eqref{eq3.3}, \eqref{A1}, \eqref{eq3.6} and \eqref{eq3.7}, it follows that
\begin{align*}
\norm{Y^n-Y^m}_{\cS^p} \to 0, \quad \mbox{as ~$n,m\to\infty$}.
\end{align*}

By Lemma \ref{le3.1}-i) and the Schwartz inequality, we get the following estimation for the first term of the right hand side of \eqref{A2}:
\begin{align*}
&E\left[ \left( \int_0^T |Y^n_s-Y^m_s||f_n(s,Y^n_s,Z^n_s)-f_m(s,Y^m_s,Z^m_s)|ds \right)^\frac{p}{2} \right] \\
&\leq C \left\{ E\left[ \left( \int_0^T |Y^n_s-Y^m_s|g_s ds \right)^\frac{p}{2} \right] + E\left[ \left( \int_0^T |Y^n_s-Y^m_s| F_{n,m}(s)ds \right)^\frac{p}{2} \right] \right\} \\
&\leq C \left[ \norm{Y^n-Y^m}_{\cS^p}^\frac{p}{2} \left\{ E\left[ \left( \int_0^T g_sds \right)^p \right] \right\}^\frac{1}{2} \right. \\
&\hspace{4em} \left. \vphantom{\left\{ E\left[ \left( \int_0^T g_sds \right)^p \right] \right\}^\frac{1}{2}} + T^\frac{p}{4}\norm{Y^n-Y^m}_{\cS^p}^\frac{p}{2} \left\{ E\left[ \left( \int_0^T\{F_{n,m}(s)\}^2 \right)^\frac{p}{2} \right] \right\}^\frac{1}{2} \right], 
\end{align*}
where $C$ is a positive constant depending only on $p$. 
Since $\norm{Y^n-Y^m}_{\cS^p}\to 0$, we obtain $\norm{Z^n-Z^m}_{\cH^p}\to 0$.
\end{proof}

By Proposition \ref{pr3.3}, we denote by $(Y,Z)$ the limit of $(Y^n,Z^n)$ in $\cS^p\times\cH^p$.
\begin{thm}
$(Y,Z)$ is an $L^p$ solution to the BSDE \eqref{BSDE}.
\end{thm}
%proof---------------------
\begin{proof}
It is already seen that
\begin{align*}
\norm{Y^n-Y}_{\cS^p}\to 0, \quad \mbox{as~~$n\to\infty$}. 
\end{align*}
By the BDG inequality, we have
\begin{align*}
\sup_{0 \leq t \leq T}\left| \int_0^t (Z^n_s-Z_s)\cdot dW_s \right| \to 0 \quad \mbox{in $L^p$, \quad as~~$n\to\infty$}. 
\end{align*}
Since $\norm{Y^n-Y}_{\cS^p}\to 0, ~\norm{Z^n-Z}_{\cH^p}\to 0$ as $n\to\infty$, we may assume 
\begin{align*}
& Y^n_t \to Y_t, \quad 0\leq t \leq T \quad \as, \\
& Z^n \to Z, \quad \mbox{$dt\times dP$-\cae}
\end{align*}
by choosing a subsequence if necessary. 
Thus, by Lemma \ref{le3.1}-iv), we get
\begin{align*}
f_n(t,Y^n_t,Z^n_t)\to f(t,Y_t,Z_t), \mbox{$\quad dt\times dP$-\cae}.
\end{align*}
Now, by Lemma \ref{le3.1}-i), we have
\begin{align*}
|f_n(t,Y^n_t,Z^n_t)| &\leq g_t+K(|Y^n_t|+|Z^n_t|). 
\end{align*}
By the H\"{o}lder inequality, $Y^n\to Y,~Z^n\to Z$ in $L^1$ with respect to $dt\times dP$, 
and then, we see that $(Y^n)_{n\geq K}$ and $(Z^n)_{n\geq K}$ are uniformly integrable with respect to $\frac{dt}{T}\times dP$. 
Hence, $(f_n(\cdot,Y^n_\cdot,Z^n_\cdot))_{n\geq K}$ is uniformly integrable with respect to $\frac{dt}{T}\times dP$. 
Thus, we get
\begin{align*}
\int_0^T |f_n(s,Y^n_s,Z^n_s)-f(s,Y_s,Z_s)|ds \to 0 \quad \mbox{in $L^1$}. 
\end{align*}
Therefore, letting $n\to\infty$ in \eqref{ABSDE}, we obtain \eqref{BSDE}.
\end{proof}

\end{document}